\documentclass[10pt]{amsart}
\usepackage{amssymb}

\theoremstyle{plain}
\newtheorem{thm}{Theorem}

\newtheorem{lem}[thm]{Lemma}

\theoremstyle{definition}

\renewcommand{\leq}{\leqslant}

\begin{document}

\title{Stasheff polytope as a sublattice of permutohedron}
\author {K. Adaricheva}
\address{Department of Mathematical Sciences, Yeshiva University,
245 Lexington ave.,
New York, NY 10016, USA}
\email{adariche@yu.edu}


\begin{abstract}
An assosiahedron $\mathcal{K}^n$, known also as Stasheff polytope, is a multifaceted combinatorial object, which, in particular, can be realized as a convex hull of certain points in $\mathbf{R}^{n}$, forming $(n-1)$-dimensional polytope\footnotemark. 

A permutahedron $\mathcal{P}^n$ is a polytope of dimension $(n-1)$ in $\mathbf{R}^{n}$ with vertices forming various permutations of $n$-element set. There exists well-known orderings of vertices of $\mathcal{P}^n$ and $\mathcal{K}^n$ that make these objects into lattices: the first known as permutation lattices, and the latter as Tamari lattices. We provide a new proof to the statement that the vertices of $\mathcal{K}^n$ can be naturally associated with particular vertices of $\mathcal{P}^n$ in such a way that the corresponding lattice operations are preserved.
In lattices terms, Tamari lattices are sublattices of permutation lattices. The fact was established in 1997 in paper by Bjorner and Wachs, but escaped the attention of lattice theorists. Our approach to the proof is based on presentation of points of an associahedron $\mathcal{K}^n$ via so-called bracketing functions. The new fact that we establish is that the embedding preserves the height of elements.
\end{abstract}

\maketitle

\section{Introduction} 
\footnotetext[1]{We use different indexing of associahedron than one that appears in the literature. This serves purposes of current paper.}
Both permutation lattices $S_n$ and Tamari lattices $T_n$ were discussed in \cite{BB} as instances of application of Newman's method of ``positive" moves performed on the words of some fixed length: $S_n$ is an outcome when the positive move switches the neighboring letters $xy$ in a $n$-letter word, while $T_n$ is build from applying shifting of brackets from $((xy)z)$ to $(x(yz))$ in bracketed sequence of $(n+1)$ letters.

Both types of lattices have their counterparts in geometry: permutation lattices correspond to permutahedra and Tamari lattices to associahedra. By the definition, a \emph{permutohedron} $\mathcal{P}^n$ of order $n$ is $(n-1)$-dimensional popytope embedded in an $n$-dimensional space, the vertices of which are formed by permuting the coordinates of vector $(1,2, \dots, n)$, see more details in \cite{K}. Thus, the permutation lattice $S_n$ is a lattice defined on vertices of the permutahedron $\mathcal{P}^n$.

\emph{An associahedron} $\mathcal{K}^n$, also known as \emph{Stasheff polytope}, can be defined as a convex polytope of dimension $n-1$, whose  vertices correspond to binary trees with $(n+1)$ leaves, or similarly, to bracketed sequences of $(n+1)$ letters, or similarly, to triangulations of a $(n+2)$-gon, see \cite{FR,GKZ,S63}. Thus, the Tamari lattice $T_n$ is a lattice on vertices of $\mathcal{K}^{n}$.

There are many realizations of Stasheff polytope known in the literature, see the references in \cite{FR}, and the combinatorial studies of associahedra is quite intensive, due to important role played by it in combinatorics, homotopy theory, study of operads, analysis of real moduli spaces and other branches of mathematics \cite{BB09,D,GKZ,L,LR98,S63,S70, S07}.  

The connections between permutohedra and associahedra were discovered in algebraic combinatorics: in \cite{BW97}, using representation of  binary trees suggested in \cite{P86,K93}, a natural order-preserving map is constructed from $\mathcal{P}^n$ onto $\mathcal{K}^{n}$ in such a way that a particular sublattice of $\mathcal{P}^n$ is lattice-isomorphic to the image. This, in particular, implies that $\mathcal{K}^{n}$ is a sublattice of $\mathcal{P}^n$. In \cite{R06}, it was observed that the mapping from \cite{BW97} is actually a lattice homomorphism, which brought to the further studies of congruences on lattices of Coxeter groups.

In our paper, we provide a new proof to the statement that Tamari lattices are sublattices of permutation lattices. 
We use the different approach, which is based on representation of Tamari lattice $T_n$ via so-called 
``bracketing" functions defined on $n$-element set. In our main theorem we show the natural embedding of bracketing functions into set of permutations on the same set that preserves the lattice operations. 

Quite surprisingly, a fundamental connection between permutation lattices and Tamari lattices were not noticed in lattice literature till now. As the result, both series of lattices $S_n$ and $T_n$ had lives of their own, but shared somewhat equal discoveries. 

First, they both had an order of non-trivial nature, thus, it took effort to show that the order is, indeed, a lattice: permutation lattices in \cite{GR} and Tamari lattices in \cite{FT} (also, shorter proof for Tamari lattices in \cite{HT}). 

Secondly, there were disjoint discoveries (\cite{DC} and \cite{G}) that both satisfy \emph{the semidistributive laws} $SD_\vee$ and $SD_\wedge$. 

Finally, an even finer structure of \emph{bounded} lattice was revealed in both: first, in \cite{Gey}, the methods of concept analysis were applied to show that $T_n$ are bounded.
\cite{C} followed with the same discovery for $S_n$, applying techniques influenced by \cite{Gey}. This result was generalized further to lattices of finite Coxeter groups in \cite{CBM}, where a new technique was applied, centered at Day's doubling construction. Using this new technique the result about boundedness of $T_n$ was also reproved in \cite{CB}.

The result of \cite{BB} puts to an end the mystery of so close a resemblance between two classes of lattices.
As we mentioned before, $T_n$ turns out to be a natural sublattice of $S_n$, thus, all universal properties of the $S_n$
are inherited in $T_n$. Of course, among such are both semidistributive laws, but also the boundedness, which is preserved for sublattices.

A new feature of embedding that was not noticed before is that $\phi: T_n\rightarrow S_n$ defined in the proof of Theorem \ref{Main} preserves the height of elements; in other words, for every element $a \in T_n$, the length of the longest chain connecting it with $0$ is the same as for its image $\phi(a)\in S_n$. In particular, $\phi$ is an $0,1$-embedding preserving atoms.

The lattice terminology and results can be found in monographs \cite{FJN,Gr}.

We note that no satisfactory description exists for the class of finite bounded lattices, which differs this class from a broader class of \emph{lower bounded} lattices. The best description of this class is still through the Day's original doubling construction: every finite bounded lattice can be obtained from a distributive lattice by a series of doublings of intervals \cite{Day}. Since the class of finite bounded lattices is a pseudo-variety, i.e. it is closed with respect to finite direct products, sublattices and homomorphic images, we suggest the following

{\bf Conjecture.} \emph{The class of finite bounded lattices $\mathbf{B_f}$ coincides with the pseudo-variety generated by all permutation lattices:} $\mathbf{B_f}=\underline{\mathbf{H}} \underline{\mathbf{S}}(S_n:n\in \omega).$

The operator $\underline{\mathbf{P_f}}$ of taking the finite direct products is not necessary here, since a finite product of permutation lattices is a sublattice of another permutation lattice.

This conjecture is raised by the author during the problem session of AMS meeting N1069, special section on Universal Algebra and Order, in Iowa City, March 18-20, 2011. 

\section{Preliminaries}

We will use the so-called \emph{position order} on the set of permutations of $n$-element set $\mathbf{n}=\{1,2,\dots,n\}$, see \cite{BB}. For each permutation $\sigma=\langle s_1,s_2,\dots, s_n\rangle$ on $\mathbf{n}$, we define $I(\sigma)$ as the set of inversions, i.e. pairs $(s_i,s_j)$, where $i<j$, while $s_i>s_j$. Then $S_n$ can be defined as the collection of $I(\sigma)$ ordered by inclusion.

As for the Tamari lattices, we use an equivalent representation that avoids direct usage of bracketing \cite{FT}, see  also Exercise 26 in \cite{Gr}. Here is a quick description of the transfer from the regular bracketing of $(n+1)$ letter word to a bracketing function.

Start from some regular bracketing of  word $x_0...x_n$ (it has also 
representation by the binary tree with leaves $x_0,...,x_n$). 
Replace consistently all brackets (AB) by A(B). As the resut, one obtains so-called right-bracketing of the word.
For example, $((x_0x_1)(x_2x_3))$ is replaced by $x_0(x_1)(x_2(x_3))$.

In the right bracketing, every letter, except $x_0$, has a unique opening bracket.
If for letter $x_i$ that unique opening bracket has the closing bracket after 
$x_j$, 
then we define $E(i)=j$, in the corresponding bracketing function.

For example, the bracketing $((ab)((cd)e))$ on $5$-element word corresponds to bracketing function $F$ on $\mathbf{4}=\{1,2,3,4\}$:
$F(1)=4,F(2)=2,F(3)=4,F(4)=4$; while the bracketing $((a((bc)d))e)$ corresponds to bracketing function $E$:
$E(1)=3,E(2)=2,E(3)=3,E(4)=4$. 

Now, define $T_n$ as the set of so-called \emph{bracketing} functions $E: \mathbf{n}\rightarrow \mathbf{n}$ that satisfy two properties:
\begin{itemize}
\item[(E1)] for each $1\leq k \leq n$, $k \leq E(k)$;
\item[(E2)] if $k \leq j\leq E(k)$, then $E(j)\leq E(k)$.
\end{itemize}

The order on these functions that forms lattice $T_n$ is point-wise: $E\leq F$, if $E(k)\leq F(k)$, for each $1\leq k\leq n$. In the example above we do have $E\leq F$.

\section{Main result}

We will prove a series of lemmas that lead to the following
\begin{thm}\label{Main}
For every $n$, $T_n\leq S_n$. Moreover, the embedding can be done preserving the height of elements.
\end{thm}

First, we want to get an easy description of a join of elements in $S_n$. For this, we will think of $I(\sigma)$ as the set of pairs $(a,b)$ in $\mathbf{n}^2$ with $a>b$. Notice that each $A=I(\sigma)$ satisfies the following properties:
\begin{itemize}
\item[(I1)] $A$ is transitive, i.e. $(a,b),(b,c) \in A$ implies $(a,c)\in A$;
\item[(I2)] if $b<c<a$ are elements from $\mathbf{n}$ and $(a,b) \in A$, then either $(a,c) \in A$, or $(c,b) \in A$. 
\end{itemize}
Indeed, the first property is evident on every $A=I(\sigma)$. As for (I2), we notice that if $a$ precedes $b$ in permutation, but $c$ does not precede $b$, then $c$ appears after $b$, hence, after $a$. It follows $(a,c) \in A$.

It turns out that properties (I1) and (I2) characterize the sets of pairs $A$ that can be represented as $I(\sigma)$. 
\begin{lem}
Suppose $A\subseteq \mathbf{n}^2$ is some set of pairs $(a,b)$ with $a>b$. If $A$ satisfies (I1) and (I2), then there exists a permutation $\sigma$ such that $A=I(\sigma)$.
\end{lem}

\begin{proof}
For each $2\leq k\leq n$ we will find a permutation $\sigma_k$ on $\mathbf{k}=\{1,2,\dots,k\}$ so that $I(\sigma_k)=A\cap \mathbf{k}^2$. In particular, $I(\sigma_n)=A$.

If $k=2$, then we take $\sigma_2=id$, if $(2,1) \not \in A$, and $\sigma_2=\langle 2,1\rangle$, if $(2,1) \in A$.

Suppose we found $\sigma_k$, $k<n$, on $\mathbf{k}$ such that $I(\sigma_k)=A\cap \mathbf{k}^2$.
We need now to extend it to $\sigma_{k+1}$ on $\mathbf{(k+1)}$. Locate the most left element $j$ among those entries of $\sigma_k$ that satisfy $(k+1,j) \in A$. If there is no such $j$, then append $(k+1)$ at the end of $\sigma_k$. 
Otherwise, put $(k+1)$ immediately left of $j$. Apparently, $A\cap \mathbf{(k+1)^2} \subseteq I(\sigma_{k+1})$.
To show the inverse inclusion, we should demonstrate that every element $x$ to the right from $k+1$ in $\sigma_{k+1}$ satisfies $(k+1,x)\in A$. If $x$ immediately follows $k+1$, then $x=j$, by the definition of $\sigma_{k+1}$, so $(k+1,x) \in A$. Thus, assume that $x$ is one of other elements to the right of $k+1$ in $\sigma_{k+1}$. If $x < j$, then $(j,x) \in A$, hence $(k+1,x)\in A$ follows from the transitivity of $A$. If $j<x$, then $j<x<k+1$ and $(k+1,j) \in A$ would imply, by (I2), that either $(k+1,x) \in A$, or $(x,j) \in A$. Apparently, the latter is not the case, due to the assumption that     
$I(\sigma_k)=A\cap \mathbf{k}^2$. Hence, $(k+1,x) \in A$, which is needed.
\end{proof}

In the next Lemma, $\vee$ is used for the join operation of lattice $S_n$, and $(X)^{tr}$ is the transitive closure of set of pairs $X$. 
\begin{lem}\label{join}
If $A_1=I(\sigma_1), A_2=I(\sigma_2)$, then $A_1\vee A_2=(A_1\cup A_2)^{tr}$.
\end{lem}
\begin{proof}
Evidently, $(A_1\cup A_2)^{tr}\subseteq A_1\vee A_2$. To show the inverse inclusion, we need to check that 
$(A_1\cup A_2)^{tr}$ satisfies conditions (I1) and (I2). (I1) evidently holds.

To verify (I2), consider $(a,b)\in (A_1\cup A_2)^{tr}$ and $b<c<a$. Without loss of generality we assume that
$b<d_1<d_2<d_3<a$ with $(a,d_3),(d_2,d_1)\in A_1$ and $(d_3,d_2),(d_1,b) \in A_2$. The following is the argument for the case $d_1 < c < d_2$, other cases are proved similarly. Since  $(d_2,d_1)\in A_1$, we would have either $(d_2,c) \in A_1$, or $(c,d_1) \in A_1$. In the first case we would have $(a,d_3),(d_3,d_2),(d_2,c) \in A_1\cup A_2$, hence, $(a,c) \in (A_1\cup A_2)^{tr}$. In the second case, $(c,d_1),(d_1,b) \in A_1\cup A_2$, hence, 
$(c,b) \in (A_1\cup A_2)^{tr}$. Thus, we got either $(a,c)$ or $(c,b) $ in $(A_1\cup A_2)^{tr}$, and the latter set satisfies (I2).
\end{proof}

Next, we will start the process of embedding $T_n$ into $S_n$. As was noted in Preliminaries, elements of $T_n$ are associated with functions $E:\mathbf{n}\rightarrow \mathbf{n}$ satisfying properties (E1) and (E2).

For each such function $E$ we create a set of pairs $I_E \subseteq \mathbf{n}^2$ as the smallest set satisfying the following rule:\\
if $k<E(k)$, for any $1\leq k\leq n$, then $(E(k),k), (E(k)-1,k), \dots, (k+1,k) \in I_E$.

\begin{lem}\label{IE}
For each element $E \in T_n$, $I_E$ satisfies  properties \emph{(I1)} and \emph{(I2)}.
\end{lem}

\begin{proof}
To prove transitivity of $I_E$, assume that $(a,b),(b,c) \in I_E$. Then, according to the definition, $b<a\leq E(b)$ and $c<b\leq E(c)$.
The latter inequality implies $E(b)\leq E(c)$, due to (E2). Thus, $c<a\leq E(b)\leq E(c)$ yields $(a,c) \in I_E$.

Now assume that $b<c<a$ and $(a,b)\in I_E$. Then $c<a<E(b)$, hence $(c,b)$ must be in $I_E$ as well. Therefore, (I2) holds.
\end{proof}

It follows from the definition of $I_E$, and it was used in the proof of the previous Lemma 
that the set of inversions $I_E$ corresponding to the bracketing functions $E$ satisfy a stronger property
(I2)* than usual property of (I2) of inversions:\\
(I2)* if $b<c<a$  and $(a,b) \in I_E$, then $(c,b) \in I_E$.

In the next Lemma, $\wedge$ stands for the meet operation in the lattice $S_n$.

\begin{lem}\label{wedge}
If $I_E$ and $I_F$ are two elements of $S_n$ that correspond to bracketing functions $E,F$, then
$I_E\wedge I_F=I_E \cap I_F$.

\begin{proof}
Evidently, $I_E\wedge I_F\subseteq I_E \cap I_F$. To show the inverse inclusion, we need to check that $I_E\cap I_F$ satisfy (I1) and (I2).
Apparently, the intersection of transitive sets is transitive. As for (I2), we note that $I_E$ and $I_F$ satisfy (I2)*, hence, $I_E\cap I_F$
satisfies it as well. In particular, it satisfies  a weaker property (I2).
\end{proof}
\end{lem}

The previous Lemma shows that $I_E\wedge I_F$ in lattice $S_n$ is a set of pairs satisfying (I1) and (I2)*.
In the next one we are showing the same for the join operation in $S_n$.

\begin{lem}\label{vee}
If $I_E$ and $I_F$ are two elements of $S_n$ that correspond to bracketing functions $E,F$, then
$I_E\vee I_F$ satisfies \emph{(I1)} and \emph{(I2)*}.
\end{lem}
\begin{proof}
According to Lemma \ref{join}, $I_E\vee I_F=(I_E\cup I_F)^{tr}$, which, of course, satisfies (I1).
So we need to check (I2)* only. Consider $(a,b)\in (I_E\cup I_F)^{tr}$ and $b<c<a$.
Without loss of generality, we may assume that $b<d_1<d_2<d_3<a$ with $(a,d_3),(d_2,d_1)\in I_E$ and $(d_3,d_2),(d_1,b) \in I_F$.
Say, $d_1<c<d_2$. Since $(d_2,d_1) \in I_E$, then $(c,d_1) \in I_E$, according to (I2)*. Then $(c,d_1),(d_1,b) \in I_E\cup I_F$, and
$(c,b) \in (I_E\cup I_F)^{tr}$, due to transitivity.
\end{proof}

Finally, we want to show that all elements of $S_n$ that satisfy (I2)* correspond to some bracketing functions, i.e., 
all of them are of form $I_E$, for some $E \in T_n$.

\begin{lem}\label{I2*}
If $I$ is an element of $S_n$ that satisfies (I2)*, then $I=I_E$, for some bracketing function.
\end{lem}
\begin{proof}
Given $I$ with (I2)*, define a function $E$ as follows: for each $1\leq k\leq n$, let $E(k)$ be the maximal element $s \in \mathbf{n}$ such that
$(s,k) \in I$, if there are such $s$, and $E(k)=k$, otherwise.
We need to check that $E$ is, indeed, the bracketing function. Evidently, $k\leq E(k)$, so (E1) holds.
Now, assume that $k < j < E(k)$. Hence, $(E(k),k)$ must be in $I$, according to the definition of $I$.
This implies $(j,k) \in I$, due to (I2)*. If $E(j) > E(k)$ would be true, then $E(j)>j$ would imply $(E(j),j) \in I$, hence, by the transitivity,
$(E(j),k) \in I$. This contradicts the choice of $E(k)$ as the largest element  $s$ of $\mathbf{n}$, for which $(s,k) \in I$.
Hence, it must be $E(j)\leq E(k)$, and we are done.
\end{proof}

The following statement will be used to establish the preservation of the height. The symbol $\prec$ stands for the
covering relation.

\begin{lem}\label{cover}
For every element $I_E \in S_n$, $I_E\not = 0$, that corresponds to a bracketing function $E$, there is another bracketing function $F$
such that $I_F \prec I_E$ in the lattice $S_n$.
\end{lem}
\begin{proof}
Since $I_E \not = 0$, there is at least one $j < n$ such that $E(j)>j$. Among those choose one such that $E(j)-j$ is minimal. 
According to (E2) and the choice of $j$, if there is $k$ with $j<k<E(j)$, then $E(k)=k$.
Consider function $F: \mathbf{n}\rightarrow \mathbf{n}$ that is defined as follows: $F(i)=E(i)$, if $i\not = j$, and $F(j)=E(j)-1$, otherwise.
It is easy to see that $F$ is a bracketing function. According to definition, $I_F=I_E\setminus \{(E(j),j)\}$.
Hence, $I_F$ and $I_E$ differ by a single pair, and $I_F\prec I_E$ in $S_n$.
\end{proof}

\emph{Proof of Theorem.}
Define mapping $\phi: T_n \rightarrow S_n$ as $\phi(E)=I_E$. Then $\phi$ is one-to-one.
Indeed, if $E\not = F$ are two elements of $T_n$, then, without loss of generality, $E(k) < F(k)$, for some
$k < n$. Then, according to the definition, $(F(k),k) \in I_F$, but $(F(k),k)\not \in I_E$.
Also, $\phi$ preserves the order. Indeed, if $E\leq F$, then $E(k) \leq F(k)$, for all $k\leq n$. Then, 
apparently, all elements of the form $(s,k)$ from $I_E$ are simultaneously in $I_F$. Hence, $I_E\subseteq I_F$.

According to Lemma \ref{I2*} and a note after Lemma \ref{IE}, $\phi(T_n)$ consists exactly of elements of $S_n$ that satisfy (I2)*.
Moreover, due to Lemmas \ref{wedge} and \ref{vee}, such elements in $S_n$ form a sublattice.
It follows that $T_n$ is isomorphic to this sublattice of $S_n$. 

Finally, it follows from Lemma \ref{cover} that, for every $\phi(E)$, there is a longest chain in $S_n$ from $0$ to $\phi(E)=I_E$ 
all whose elements are in the range of $\phi$. This confirms that the height of elements in $T_n$ are preserved under $\phi$.
In particular, $\phi$ is $0,1$-embedding that maps atoms to atoms.
\emph{End of proof of Theorem.}    

\emph{Acknowledgements.} Since placing this preprint in public domain on January 7, 2011, the author received several pointers to the publications in algebraic combinatorics, in particular, to critical paper \cite{BW97}. We would like  to thank N.~Reading, H.~Thomas, F.~Sottile, J.~Stasheff and J.D.H.~Smith for their attention and comments. I hope this preprint will serve as a connector to currently well disjoint studies of same objects and will advance communications across the fields. The text of this preprint may be used in the future chapter on semidisributive lattices, in the second volume of General Lattice Theory (G.~Gr\"atzer and F.~Wehrung, Eds), 3d edition (to appear).

\end{document}